\documentclass[a4papersize]{article}

\usepackage{amsmath,amsthm, amsxtra,amssymb,latexsym, amscd}
\usepackage[mathscr]{eucal}
\newtheorem{theorem}{Theorem}[section]
\newtheorem{corollary}[theorem]{Corollary}
\newtheorem{notation}[theorem]{Notation}
\newtheorem{lemma}[theorem]{Lemma}
\newtheorem{proposition}[theorem]{Proposition}
\newtheorem{example}[theorem]{Example}

\newtheorem{definition}[theorem]{Definition}
\newtheorem{remark}[theorem]{Remark}
\DeclareMathOperator{\nSCM}{nSCM}
\DeclareMathOperator{\depth}{depth}

\DeclareMathOperator{\Ann}{Ann}
\DeclareMathOperator{\docao}{ht}
\DeclareMathOperator{\Ass}{Ass}
\DeclareMathOperator{\Defi}{Defi}
\DeclareMathOperator{\Supp}{Supp}

\DeclareMathOperator{\Var}{Var}
\DeclareMathOperator{\nCM}{nCM}
\DeclareMathOperator{\Tor}{Tor}
\DeclareMathOperator{\f-depth}{f-depth}

\DeclareMathOperator{\Att}{Att}
\DeclareMathOperator{\E-depth}{E-depth}
\DeclareMathOperator{\Ef-depth}{Ef-depth}
\DeclareMathOperator{\Spec}{Spec}
\DeclareMathOperator{\p}{\frak p}

\begin{document}
\large
\centerline{\Large {\bf   ON THE PERTUBATIONS OF E-DEPTH, EF-DEPTH AND}}
\smallskip
\centerline{\Large {\bf    SEQUENTIALLY COHEN-MACAULAY LOCUS}}

\medskip
\vskip 0.7cm

\medskip

\vskip 0.7cm
\centerline { DOAN TRUNG CUONG}
\centerline {Institute of Mathematics}
\centerline{Vietnam Academy of Science and Technology, Hanoi, Vietnam}
\centerline {E-mail: dtcuong@math.ac.vn}
\vskip 0.4cm
\centerline { LE THANH NHAN}
\centerline {Institute of Mathematics}
\centerline{Vietnam Academy of Science and Technology, Hanoi, Vietnam}
\centerline {E-mail: nhanlt2014@gmail.com; lttnhan@math.ac.vn}
\vskip 0.4cm
\centerline { NGUYEN XUAN LINH}
\centerline {  Hanoi  University of Civil Engineering}
\centerline{Hanoi, Vietnam}
\centerline {E-mail: linhnx@huce.edu.vn}

\medskip
\vskip 0.7cm

\noindent{\bf Abstract} {\footnote{ {\it{Key words and phrases: }} Sequential sequence, sequential f-sequence, E-depth, Ef-depth, sequentially Cohen-Macaulay locus. \hfill\break
  {\it{2000 Subject  Classification: }}  13E05, 13C14, 13D45\\
   The work is supported by the Vietnam National Foundation for Science and Technology Development (Nafosted) under grant number 101.04-2023.31.}.  Let $(R, \frak m )$ be a Noetherian local ring, $I$ an ideal of $R$ and $M$ a finitely generated $R$-module.  In this paper, we define and study $\E-depth$ and $\Ef-depth$ of $M$ in $I$. We prove that $\E-depth$ (resp. $\Ef-depth$ under  mild conditions) of $M$ in $I$ is the common length of all maximal sequential sequences (resp. maximal sequential f-sequences) of $M$ in $I$. We show that $\E-depth$ and $\Ef-depth$ do not decrease under small pertubations.  We describe the non sequentially Cohen-Macaulay locus $\nSCM(M)$ of $M$ in terms of support and non Cohen-Macaulay locus of deficiency modules of $M$. Using this description, we prove that if $x_1, \ldots , x_r$ is a sequential f-sequence then there exists an integer $N>0$ such that 
$$\dim\nSCM(M/(x_1+\epsilon_1, \ldots , x_r+\epsilon_r)M)\leq \dim\nSCM(M/(x_1, \ldots , x_r)M)$$ for all $\epsilon_1, \ldots , \epsilon_r\in\frak m^N$.       

\section{Introduction}

We first fix some notations and assumptions.

\begin{notation} \label{N:1} {\rm Throughout this paper, let $(R,\frak m)$ be  a Noetherian local ring, $M$  a finitely generated $R$-module of dimension $d$. Let $\widehat R$ and $\widehat M$ denote the $\frak m$-adic completion of $R$ and $M$, respectively. For an ideal   $I= (x_1,\ldots , x_n)$ of $R$  and a $n$-tuple $\underline {\epsilon}=(\epsilon_1, \ldots , \epsilon_n)$ of elements in $\frak m$, we set $I(\underline {\epsilon})= (x_1+\epsilon_1,\ldots , x_n+\epsilon_n).$ If $\epsilon_1,\ldots , \epsilon_n\in\frak m^N$ for $N\gg 0$ then  $I(\underline {\epsilon})$ is called a  {\it small perturbation} of the ideal $I$. If $R$ is a quotient of a Gorenstein local ring, denote by $K^i(M)$  the $i$-th deficiency module of $M$ for $i\in\Bbb N$. Set $\Defi(M)=\{i\geq 1\mid K^i(\widehat M)\neq 0\}$.}
\end{notation}

It is well-known that $M$ is Cohen-Macaulay if and only if there exists a system of parameters (s.o.p for short) of $M$ that is an $M$-sequence.  The notion of filter regular sequence (f-sequence for short) was introduced by N. T. Cuong, P. Schenzel, N. V. Trung \cite{CST} as a natural extension of the regular sequence concept.   If $R$ is a quotient of a Cohen-Macaulay local ring then  $M$ is generalized Cohen-Macaulay if and only if  each s.o.p of $M$ is an f-sequence, cf. \cite{CST}. The class of sequentially Cohen-Macaulay modules is an important extension of the class of Cohen-Macaulay modules, which was introduced by R. Stanley 1996 (in graded setting) and P. Schenzel 1998 (in local setting), and is mainly devoted to studying the modules that are not necessarily unmixed.  The class of  sequentially generalized Cohen-Macaulay modules was defined in \cite{CN} in a natural way.   Let $H^0_{\frak m}(M)=D_t\subset \ldots \subset D_1\subset D_0=M$ be the dimension filtration of $M$, i.e. each $D_{i+1}$ is the largest submodule of $M$ of dimension less than $\dim_R(D_i)$. We say that  $M$  is {\it sequentially Cohen-Macaulay} (resp. {\it sequentially generalized Cohen-Macaulay})  if each quotient  $D_i/D_{i+1}$ is  Cohen-Macaulay (resp. generalized Cohen-Macaulay). Sequentially  Cohen-Macaulay modules and sequential generalized Cohen-Macaulay modules have been studied extensively in the literature from various aspects, see \cite{St, Sch, HS, CN, CC, NDC, CSt, CSSS, LN}. 

Let us mention that for an $M$-regular element $x$, P. Schenzel \cite[Theorem 4.7]{Sch} proved that the property of being a sequentially  Cohen-Macaulay module on $M$ is inherited by $M/xM$, however in general the sequential Cohen-Macaulayness does not lift from $M/xM$ to $M$, see Example \ref{E:2}(b).  Therefore, the notions of sequential sequence and sequential f-sequence (see Definition \ref{D:1}) were introduced  with the purpose of establishing an analogue for the sequentially Cohen-Macaulay modules and sequentially generalized Cohen-Macaulay modules of the above well-known characterizations of Cohen-Macaulay modules and generalized Cohen Macaulay modules, where the roles of regular sequences and filter regular sequences are respectively replaced by that of sequential sequences and sequential f-sequences, see \cite[Theorems 1.2, 1.3]{LN}. It is important that $M$  is sequentially  Cohen-Macaulay if and only if so is $M/xM$ for any sequential element $x$ of $M$, see Lemma \ref{L:1a}(a).    
    
On the other hand, the question of which properties and numerical invariants are preserved under small perturbations has attracted the interest among mathematicians, see \cite{Sa, Ei, CS, ST, HT, MS, MQS, QT, QTr}. This question is related to the problem of approximating analytic singularities by algebraic singularities. D. Eisenbud \cite{Ei} studied homology of complexes under a pertubation and showed that Euler characteristic and depth can be preserved. There are some cases in which the rings $R/I(\underline {\epsilon})$ and $R/I$ are
isomorphic for all sufficiently small pertubations $I(\underline {\epsilon})$ of $I$, see \cite{Sa, CS, MS}. The behaviour under perturbation of Hilbert-Samuel functions has been investigated in \cite{ST, MQS, QT, QTr}. If $x_1, \ldots , x_r$ is an f-sequence and $R$ is an equidimensional quotient of a Cohen-Macaulay local ring then there exists an integer $N>0$ such that there is an inequality on dimensions of non Cohen-Macaulay loci 
$$\dim\nCM(R/(x_1+\epsilon_1, \ldots , x_r+\epsilon_r))\leq \dim\nCM(R/(x_1, \ldots , x_r))$$  
 for all $\epsilon_1, \ldots , \epsilon_r\in\frak m^N$, cf. \cite[Theorem 4.6]{MQS}.

In this paper,  we define $\E-depth_R(I,M)$ and $\Ef-depth_R(I,M)$ of $M$ in  $I$ in terms of  the depth and f-depth of the deficiency modules $K^i(\widehat M)$ in $I$ respectively (Definition \ref{D:2}). We show that  $\E-depth_R(I,M)$ is the common length of all maximal sequential sequences of $M$ in $I$; and if $\dim_R(M/IM)\geq 2$ then  $\Ef-depth_R(I,M)$ is a finite number which is equal to the common length of all maximal sequential f-sequences of $M$ in $I$. We prove that $\E-depth$ and $\Ef-depth$ of $M$ in suficiently small pertubations $I(\underline {\epsilon})$ of $I$ are  at least $\E-depth_R(I,M)$ and $\Ef-depth_R(I,M)$, respectively (Theorem \ref{T:1bs}). We describe the non sequentially Cohen-Macaulay locus $\nSCM(M)$ of $M$ in terms of the support and the non Cohen-Macaulay locus of deficiency modules of $M$ (Theorem \ref{T:1}). Using this description, we prove that if $x_1, \ldots , x_r$ is a sequential f-sequence of $M$, then there exists an integer $N>0$ such that 
$$\dim\nSCM(M/(x_1+\epsilon_1, \ldots , x_r+\epsilon_r)M)\leq \dim\nSCM(M/(x_1, \ldots , x_r)M)$$ for all $\epsilon_1, \ldots , \epsilon_r\in\frak m^N$ (Theorem \ref{T:2}). We also give some examples to clarify the results.

In the next section, after giving some preliminaries on sequential sequences and sequential f-sequences, we show that sequential sequences and sequential f-sequences are preserved under small pertubations, then we prove Theorem \ref{T:1bs}. The proofs of the main results, Theorems \ref{T:1}, \ref{T:2}, are presented in the last section. 

\section{Sequential sequences and sequential f-sequences}

    The notions of sequential sequence and sequential f-sequence were introduced in \cite{LN} in order to characterize  the sequential Cohen-Macaulay modules and sequential generalized Cohen-Macaulay modules. These notions were defined in terms of attached primes of Artinian local cohomology modules $H^i_{\frak m}(M)$. Following I. G. Macdonald \cite{Mac}, every Artinian $R$-module $A$ has a minimal secondary representation $A=A_1+\ldots +A_n$, where $A_i$ is $\frak p_i$-secondary,  $A_i$ is not redundant and $\frak p_i\neq \frak p_j$ for all $i\neq j.$ The set $\{\frak p_1, \ldots , \frak p_n\}$ is independent of the  choice of minimal secondary representation of $A$.  This set is called the  set of {\it attached primes} of $A$, denoted by $\Att_RA$.

  \begin{definition} \label{D:1} 
  	{\rm (a)} {\rm  An element $x\in \frak m$ is called an {\it $M$-sequential element} if $\dim_R(M)>0$ and $x\notin\frak p$ for all $ \frak p\in\bigcup_{i=1}^d\Att_RH^i_{\frak m}(M)$. A sequence  $x_1, \ldots , x_r$ of elements in $\frak m$ is called an {\it $M$-sequential sequence}  if $x_i$ is an $M/(x_1, \ldots , x_{i-1})M$-sequential element for all $i\leq r$}.
  	
  	{\rm (b)} {\rm An element $x\in \frak m$ is said to be an {\it $M$-sequential f-element}  if $x\notin\frak p$ for all prime ideals $ \frak p\in\bigcup_{i=2}^d\Att_RH^i_{\frak m}(M)$ with $\dim (R/\frak p)>0$. A sequence $x_1, \ldots , x_r$ of elements in $\frak m$ is called an {\it $M$-sequential f-sequence} if $x_i$ is an $M/(x_1, \ldots , x_{i-1})M$-sequential f-element for all $i\leq r$}.
  \end{definition}
  
  We need recall some basic properties for attached primes of Artinian modules that will be used later, see \cite[8.2.4, 8.2.5]{BS}, \cite[Corollary 4.2]{CNN}. For an ideal $I$ of $R$, denote by ${\Var}(I)$ the set of all prime ideals of $R$ containing $I$.  

\begin{lemma}\label{L:1} Let $A$ be an Artinian $R$-module.   Then 
\begin{itemize}
\item[{\rm (a)}]  $\min \Att_RA=\min {\Var} (\Ann_RA).$ 
\item[{\rm (b)}] $A$ has a natural structure as an Artinian $\widehat R$-module and $$\Att_RA=\{\frak P\cap R\mid \frak P\in \Att_{\widehat R}A\}.$$
\item[{\rm (c)}] Let $i\in\Bbb N$. If $\frak p\in\Ass_R(M)$ with $\dim (R/\frak p)=i$ then $\frak p\in\Att_RH^i_{\frak m}(M)$.
\item[{\rm (d)}] Let $R$ be a quotient of a Cohen-Macaulay local ring. Let $i\in\Bbb N$ and $\frak p\in \Att_RH^i_{\frak m}(M).$ Then $\dim (R/\frak p)\leq i.$ Moreover, $\dim (R/\frak p)=i$ if and only if $\frak p\in\Ass_R(M).$  
\end{itemize}
\end{lemma}

In the following remark, we see that sequential sequences, sequential f-sequences are preserved under $\frak m$-adic completion. We also find some relations between the notions of sequential sequence, sequential f-sequence and the more known concepts of sequences.  Following \cite{CST}, an element $x\in\frak m$ is a {\it filter regular element} (f-element for short) if $x\notin\frak p$ for all $\p\in \Ass_R(M)$ satisfying $\dim (R/\frak p)>0$. Following \cite{Nh}, an element  $x\in\frak m$ is said to be a {\it generalized regular element} if $x\notin\frak p$ for all $\p\in \Ass_R(M)$ satisfying $\dim (R/\frak p)>1$. A sequence $x_1, \ldots , x_r$ of elements in $\frak m$ is called an {\it f-sequence} (resp. {\it generalized regular sequence}) of $M$ if  $x_i$ is an f-element (resp. generalized regular element) of  $M/(x_1, \ldots , x_{i-1})M$ for all $i\leq r$.

\begin{remark} \label{R:1} {\rm (a) Let $x_1, \ldots , x_r\in\frak m$. It follows by Lemma \ref{L:1}(b) that  $x_1, \ldots , x_r$  is an $M$-sequential sequence (resp. $M$-sequential f-sequence) if and only if it is an $\widehat M$-sequential sequence (resp. $\widehat M$-sequential f-sequence).
		
(b) Since  $\Ass_R(M)\subseteq \underset{i\leq d}{\bigcup}\Att_RH^i_{\frak m}(M)$ by Lemma \ref{L:1}(c), it follows that each sequential sequence of $M$ is an f-sequence; each sequential f-sequence of $M$ is a  generalized regular sequence. Note that sequential sequences (resp. sequential f-sequences) are not necessarily regular sequences (resp. f-sequences); regular sequences (resp. f-sequences) are not necessarily sequential sequences (resp. sequential f-sequences), see \cite[Remark 2.5]{LN}.}
\end{remark}

\begin{lemma} \label{L:1a} Suppose that $R$ is a quotient of a Gorenstein local ring. Let $\Defi(M)$ be defined as in Notation \ref{N:1}. Then
	\begin{enumerate}
\item[{\rm (a)}]  An element $x\in\frak m$ is an $M$-sequential element if and only if $x$ is $K^i(M)$-regular for all $i\in\Defi(M).$ If $x$ is an $M$-sequential element then $K^{i+1}(M)/xK^{i+1}(M)\cong K^i(M/xM)$ for all $i>0$ and there is an exact sequence  
$$0\to K^1(M)/xK^1(M)\to K^0(M/xM)\to (0:_{K^0(M)}x)\to 0.$$ In particular, $M$  is sequentially  Cohen-Macaulay if and only if so is $M/xM$.
\item[{\rm (b)}] A sequence  $x_1, \ldots , x_r$ of elements in $\frak m$  is an $M$-sequential sequence if and only if $x_1, \ldots , x_{\min\{i,r\}}$ is $K^i(M)$-sequence for all $i\in\Defi(M)$. 
\item[{\rm (c)}] An element $x\in\frak m$ is an $M$-sequential f-element if and only if $x$ is an f-element of $K^i(M)$ for all $i>1$.  If $x$ is an $M$-sequential f-element then  $\dim_{R}(0:_{K^i(M)}x)\leq 0$ for all $i>1$ and  there exists for each $i>0$ an exact sequence $$0\to K^{i+1}(M)/xK^{i+1}(M)\to  K^i(M/xM)\to (0:_{K^i(M)}x)\to 0.$$  
\item[{\rm (d)}] A sequence  $x_1, \ldots , x_r$ of elements in $\frak m$  is an $M$-sequential f-sequence if and only if $x_1, \ldots , x_{\min\{i-1,r\}}$ is an f-sequence of $K^i(M)$ for all $i>1$.  
\end{enumerate}
\end{lemma}

\begin{proof} The statement (a) was proved in \cite[Lemma 3.6]{C2025}. 

 The statement (b) follows by assertion (a) and induction on $r$.

 (c) It follows by   \cite[10.2.20, 11.2.6]{BS} that $\Att_{R}H^i_{\frak m}(M)=\Ass_{R}K^i(M)$ for all integers $i$. So, $x$ is an $M$-sequential f-element if and only if  $x$ is an f-element of $K^i(M)$ for all $i>1.$ Suppose that $x$ is an $M$-sequential f-element. Then it is a generalized regular element of $M$ by Remark \ref{R:1}(b). So,  there exists by  \cite[Lemma 2.4]{DN} an exact sequence 
$$0\to H^i_{\frak m}(M)/xH^i_{\frak m}(M)\to  H^i_{\frak m}(M/xM)\to (0:_{H^{i+1}_{\frak m}(M)}x)\to 0$$ for all $i\geq 1$. Hence, we have the induced exact sequence
 $$0\to K^{i+1}(M)/xK^{i+1}(M)\to  K^i(M/xM)\to (0:_{K^i(M)}x)\to 0$$ for all $i\geq 1.$ As $x$ is an f-element of $K^i(M)$, we have $\dim_R(0:_{K^i(M)}x)\leq 0$  for all $i\geq 2.$  

The statement (d) follows by assertion (c) and induction on $r$. 
\end{proof}

It is natural to ask whether sequential sequences are preserved under permutations. The  answer is negative in general.  Below we show that even $R$ is sequentially Cohen-Macaulay, a permutation of a sequential sequence need not be  a sequential sequence.
 
\begin{example} \label{E:1a} {\rm Let $R=K[[x, y,z]]/(x)\cap (y,z)$, where $K[[x, y,z]]$ is the formal power series ring over a field $K$. Set $\frak m=(x, y, z)R.$  It is clear that  $R$ is sequentially Cohen-Macaulay. We have $\Att_R H^2_{\frak m}(R)=\{xR\}$, $\Att_R H^1_{\frak m}(R)=\{(y,z)R\}.$ Hence,  $x + y, z$ is an $R$-sequential sequence, but $z, x+y$ is not an $R$-sequential sequence}.
\end{example}

A sequence $x_1, \ldots , x_r\in\frak m$ is called an {\it unconditioned  sequential sequence} (resp. {\it unconditioned f-sequence}) of $M$ if $x_{\sigma(1)},\ldots, x_{\sigma(r)}$ is a  sequential sequence (resp. f-sequence) of $M$ for all permutations $\sigma\in\Bbb S_r.$  
 Now, we give a criterion for a sequential sequence to be an unconditioned  sequential sequence. 

\begin{proposition} \label{P:1b} Let $x_1, \ldots , x_r\in\frak m$ be a sequential sequence of $M$. Then $x_1, \ldots , x_r$ is an unconditioned sequential sequence of $M$ if and only if it is an unconditioned f-sequence. 
\end{proposition}

\begin{proof} By Remark \ref{R:1}(a), we can assume that $R$ is complete. By Remark \ref{R:1}(b), if $x_1, \ldots , x_r$ is an unconditioned $M$-sequential sequence then it is an unconditioned f-sequence. 

Assume that $x_1, \ldots , x_r$ is an unconditioned f-sequence of $M$. We prove by induction on $r$ that $x_1, \ldots , x_r$ is an unconditioned $M$-sequential sequence.  The case of $r=1$ is clear. Let $r=2$. First we prove that $x_2$ is an $M$-sequential element.  Let $\Defi(M)$ be defined as in Notation \ref{N:1} and let $i\in\Defi(M)$. By Lemma \ref{L:1a}(a), we need to prove that $x_2$ is $K^i(M)$-regular. Set $L=(0:_{K^i(M)}x_2)$. Let $z\in L$. Then $x_2z=0\in x_1K^i(M).$ Hence $z\in (x_1K^i(M):_{K^i(M)}x_2).$ Since $i\in\Defi(M)$, it follows by Nakayama and Lemma \ref{L:1a}(a) that  $K^{i-1}(M/x_1M)\neq 0$. We divide into two cases.

\noindent{\it Case 1: $i\geq 2.$} Then $i-1\in\Defi(M/x_1M)$. Since $x_2$ is an $M/x_1M$-sequential element, $x_2$ is $K^{i-1}(M/x_1M)$-regular. Hence $x_2$ is $K^i(M)/x_1K^i(M)$-regular by Lemma \ref{L:1a}(a). So, $z\in (x_1K^i(M):_{K^i(M)}x_2)=x_1K^i(M).$  Write $z=x_1z'$ for some $z'\in K^i(M).$ We have $0=x_2z=x_2x_1z'.$ Hence $x_2z'\in (0:_{K^i(M)}x_1)=0$ because $x_1$ is $K^i(M)$-regular. So,  $x_2z'=0$. Hence $z'\in (0:_{K^i(M)}x_2)=L.$ Therefore, each $z\in L$, there exists $z'\in L$ such that $z=x_1z'.$ Hence, $x_1L=L.$ So, $L=0$ by Nakayama's Lemma. Hence $x_2$ is $K^i(M)$-regular. 

\noindent{\it Case 2: $i=1.$} As $x_1$ is $K^1(M)$-regular, $K^1(M)$ is Cohen-Macaulay of dimension $1$. Let $\frak p\in\Ass_RK^1(M).$ Then $\dim (R/\frak p)=1.$ Hence $\frak p\in\Ass_R(M)$.  As $x_1, x_2$ is an $M$-sequential sequence, we get by Remark \ref{R:1}(b) that $x_1, x_2$ is an f-sequence of $M$. So, $x_2, x_1$ is an f-sequence of $M$ by our assumption. Hence $x_2\notin\frak p$. Therefore, $x_2$ is $K^1(M)$-regular. 

Next we prove that $x_1$ is an $M/x_2M$-sequential element. Let $i\in\Defi(M/x_2M).$  By Lemma \ref{L:1a}(a), we need  to show that $x_1$ is $K^i(M/x_2M)$-regular. As $x_2$ is an $M$-sequential element,  $K^i(M/x_2M)\cong K^{i+1}(M)/x_2K^{i+1}(M)$  by Lemma \ref{L:1a}(a). Hence $K^{i+1}(M)\neq 0$.   Let $z\in (x_2K^{i+1}(M):_{K^{i+1}(M)}x_1).$ Then $x_1z=x_2y$ for some $y\in K^{i+1}(M)$. It follows that  $y\in (x_1K^{i+1}(M):_{K^{i+1}(M)}x_2).$ Since $x_1$ is $M$-sequential element, we have by Lemma \ref{L:1a}(a) that $K^i(M/x_1M)\cong K^{i+1}(M)/x_1K^{i+1}(M)\neq 0.$  As $x_2$ is an $M/x_1M$-sequential element, $x_2$ is $K^{i+1}(M)/x_1K^{i+1}(M)$-regular. Hence $y\in (x_1K^{i+1}(M):_{K^{i+1}(M)}x_2)=x_1K^{i+1}(M).$ So we can write $y=x_1y'$ for some $y'\in K^{i+1}(M).$ Hence $x_1z=x_2x_1y'$ and hence  $x_1(z-x_2y')=0$. So, $z-x_2y'\in (0:_{K^{i+1}(M)}x_1)=0$ as $x_1$ is an $M$-sequential element. Hence, we have $z=x_2y'\in x_2K^{i+1}(M).$ So, $x_1$ is $K^{i+1}(M)/x_2K^{i+1}(M)$-regular. Hence $x_1$ is $K^i(M/x_2M)$-regular. The result is true for $r=2.$

Let $r>2.$ Note that any permutation is a composition of transpositions of adjacent indices. So, it is enough to show that $x_1,\ldots, x_{i-2},x_i,x_{i-1},x_{i+1},\ldots , x_r$ is an $M$-sequential sequence for all $i$. This follows from the case $r=2$ to the module $M/(x_1,\ldots,x_{i-2})$ and its sequential sequence $x_{i-1}, x_i$.
\end{proof}

 We recall the results by D. Eisenbud \cite{Ei}, C. Huneke and V. Trivedi \cite[Remarks 1.12(1)]{HT} on the preservation of regular sequences and f-sequences under small perturbation.   

\begin{lemma} \label{L:1c}  If $x_1,\ldots, x_r\in\frak m$ is a regular sequence (resp.  f-sequence) of $M$ then there exists an integer $N>0$ such that $x_1+\epsilon_1,\ldots, x_r+\epsilon_r$ is a regular sequence (resp.  f-sequence) of $M$ for all $\epsilon_1,\ldots , \epsilon_r\in\frak m^N$.
\end{lemma}

Now we show that sequential sequences and  sequential f-sequences are preserved under small pertubations.

\begin{proposition} \label{P:1}  If $x_1, \ldots , x_r\in\frak m$ is a sequential sequence (resp. sequential f-sequence) of $M$ then there exists an integer $N>0$ such that $x_1+\epsilon_1, \ldots , x_r+\epsilon_r$ is a sequential sequence (resp. sequential f-sequence) of $M$ for all $\epsilon_1, \ldots , \epsilon_r\in\frak m^N.$ 
\end{proposition}

\begin{proof} By Remark \ref{R:1}(a), we can assume that $R$ is complete. 

Assume that $x_1, \ldots , x_r$  is a sequential sequence of $M$. Let $\Defi(M)$ be defined as in Notation \ref{N:1} and let  $i\in\Defi(M)$.  Then $x_1, \ldots , x_{\min\{i,r\}}$ is a $K^i(M)$-sequence by Lemma \ref{L:1a}(b). By Lemma \ref{L:1c}, there exists an integer $N_i>0$ such that $x_1+\epsilon_1, \ldots , x_{\min\{i,r\}}+\epsilon_{\min\{i,r\}}$ is a $K^i(M)$-sequence  for all $\epsilon_1,\ldots , \epsilon_{\min\{i,r\}}\in\frak m^{N_i}$. Set $N=\max\{N_i\mid i\in\Defi(M)\}.$ Then, for any $\epsilon_1,\ldots , \epsilon_r\in\frak m^{N},$ the sequence  $x_1+\epsilon_1, \ldots , x_r+\epsilon_r$ is an $M$-sequential sequence by Lemma \ref{L:1a}(b). The rest statement follows by using Lemma \ref{L:1a}(d),  Lemma \ref{L:1c} with the similar arguments.  
\end{proof}

 Caviglia-Stefani \cite{CSt} defined the $\E-depth$ of a graded module over a polynomial ring  and use it to characterize sequential Cohen-Macaulay modules. The first author \cite{C2025} made some slight adjustments to define $\E-depth$ in both graded case and local case, then found the following relation between $\E-depth$ and the lengths of sequential sequences.   
  
\begin{theorem} {\rm (See \cite[Theorems 3.8, 3.13]{C2025})}. \label{L:1cbs} Let $\Defi(M)$ be defined as in Notation \ref{N:1}. Then each sequential sequence of $M$ can be extended to a maximal one, all maximal sequential sequences of $M$  have the same length. This common length is equal to $\E-depth_R(M)$ defined as follows
$$\E-depth_R(M)=\max\{r\leq d\mid  \depth_{\widehat R}(K^i(\widehat M))\geq \min\{i,r\}\ \text{for all}\ i\in\Defi(M)\}.$$
In particular, $M$ is sequentially Cohen-Macaulay if and only if $\E-depth_R(M)=\dim_R(M)$ and $R/\Ann_R(M)$ is a quotient of a Cohen-Macaulay local ring.  
\end{theorem}

Theorem \ref{L:1cbs} suggests the definition of $\E-depth$ and $\Ef-depth$ of $M$ in an ideal $I$.

\begin{definition} \label{D:2} {\rm For a proper ideal $I$ of $R$, the $\E-depth_R(I,M)$ and $\Ef-depth_R(I,M)$ of $M$ in $I$ are defined as follows:} 
\begin{align} \E-depth_R(I, M)&:=\max\{r\leq d\mid \depth_{\widehat R}(I\widehat R, K^i(\widehat M))\geq \min\{i,r\} \ \text{\rm for all}\ i\in\Defi(M)\}.\notag\\
\Ef-depth_R(I,M)&:=\sup\{r\in\Bbb N\mid \f-depth_{\widehat R}(I\widehat R, K^i(\widehat M))\geq \min\{i-1,r\}\ \text{\rm for all}\ i>1\}.\notag\end{align}
\end{definition} 

 Here, $\f-depth_R(I,M)$ was introduced by R. Lu and Z. Tang \cite{LT}, which was defined to be the supremum of the lengths of f-sequences of $M$ in $I$. Note that $$\f-depth (I,M)=\inf\{i\in\Bbb N\mid H^i_I(M)\ \text{is not Artinian}\}.$$ If $\dim(M/IM)>0$ then $\f-depth_R(I,M)<\infty$. In this case, each f-sequence of $M$ in $I$ can be extended to a maximal one,  all maximal f-sequences of $M$ in $I$ have the same length and this common length is equal to $\f-depth_R(I,M)$.  If $\dim(M/IM)\leq 0$ then there exists an f-sequence of $M$ in $I$ of length $r$ for any $r\in\Bbb N$, in this case $\f-depth(I,M)=\infty.$  

The  depth and f-depth do not decrease under small pertubations as shown below.

\begin{lemma} \label{L:1d} Let $I$ be a proper ideal of $R$. Let $x_1, \ldots , x_n$ be a system of generators of $I$.   Then there exists an integer $N>0$ such that $\depth (I, M)\leq \depth(I(\underline{\epsilon}),M)$ and $\f-depth (I, M)\leq \f-depth(I(\underline{\epsilon}),M)$ for any $n$-tuple $\underline{\epsilon}=(\epsilon_1, \ldots , \epsilon_n)$ of elements  in $\frak m^N.$ 
\end{lemma}

\begin{proof}  Set $\depth (I, M)=r.$ Let $y_1, \ldots , y_r$ be an $M$-sequence in $I$. For each $i\leq r$, write $y_i=\sum_{j=1}^na_{ij}x_j$, where $a_{ij}\in R$. By Lemma \ref{L:1c}, there exists an integer $N>0$ such that $y_1+\epsilon_1, \ldots , y_r+\epsilon_r$ is an $M$-sequence for all $\epsilon_1, \ldots , \epsilon_r\in\frak m^N$. 

Now, let $\underline \epsilon =(\epsilon_1, \ldots , \epsilon_n)$, where $\epsilon_i\in \frak m^N$ for all $i\leq n$. Set $\epsilon'_i=\sum_{j=1}^na_{ij}\epsilon_j$ for $i\leq r.$  Then $$y_i+\epsilon'_i=\sum_{j=1}^na_{ij}(x_j+\epsilon_j)\in I(\underline \epsilon)$$ for all $i\leq r$. Note that the sequence $y_1+\epsilon'_1, \ldots , y_r+\epsilon'_r$ is an $M$-sequence because $\epsilon'_i\in\frak m^N$ for $i\leq r$. Hence $\depth(I(\underline{\epsilon}),M)\geq r$.  

 Now we prove that the f-depth is preserved under small pertubations. Let $r\in\Bbb N$ and suppose that $y_1, \ldots , y_r$ is an f-sequence of $M$ in $I$. Then there exists by Lemma \ref{L:1c} an integer $N>0$ such that $y_1+\epsilon_1, \ldots , y_r+\epsilon_r$ is an f-sequence of $M$ for all $\epsilon_1, \ldots , \epsilon_r\in\frak m^N.$ Now for given $\epsilon_1, \ldots , \epsilon_n\in\frak m^N,$ by similar arguments as in the above proof for the depth, we can find elements  $\epsilon'_1, \ldots , \epsilon'_r\in \frak m^N$ such that $y_1+\epsilon'_1, \ldots , y_r+\epsilon'_r$ is an f-sequence of $M$ in $I(\underline \epsilon).$ So, if $\f-depth (I, M)<\infty$ then for $r=\f-depth (I, M)$ we have  $\f-depth_R(I(\underline \epsilon),M)\geq r.$ Assume that $\f-depth_R(I,M)=\infty.$ Then for $r=d$, we get an f-sequence $y_1+\epsilon'_1, \ldots , y_d+\epsilon'_d$ of $M$ in $I(\underline \epsilon).$ Hence $y_1+\epsilon'_1, \ldots , y_d+\epsilon'_d, y_{d+1}, \ldots , y_t$ is an f-sequence of $M$ in $I(\underline \epsilon)$ for all $t>d$ and all elements $y_{d+1}, \ldots , y_t\in I(\underline \epsilon).$ Therefore, $\f-depth_R(I(\underline \epsilon),M)=\infty$.
\end{proof}

The following example shows that for a given  $N>0$, the $\depth (I(\underline{\epsilon}), M)$ may be strictly bigger than $\depth (I,M)$ for some $n$-tuple $\underline{\epsilon}=(\epsilon_1, \ldots , \epsilon_n)$ of elements  in $\frak m^N$. 

\begin{example} \label{E:1b} {\rm Let $M=R/(u_1) \cap (u_2, u_3)$, where $R=K[[u_1, u_2, u_3, u_4]]$ is the formal power series ring over a field $K$. Let $I=(u_1,u_2)$. It is obvious that $u_1-u_2$ is an $M$-regular element in $I$. Hence   $\depth (I, M)>0.$  It is clear that $\depth (I, R/u_1R)>0$, $\depth (I, R/(u_2,u_3)R)>0$ and $\depth (I, R/(u_1,u_2,u_3)R)=0$. Therefore, from the exact sequence
$$0 \to M \to R/(u_1) \oplus R/(u_2, u_3) \to R/(u_1, u_2, u_3) \to 0,$$ 
we get $H^1_I(M) \neq 0.$ Hence $\depth (I, M)=1$.
		 
Let $N>0$ be an integer. Set $\underline{\epsilon}=(u_4^N, u_4^N).$ Then  $I(\underline{\epsilon})=(u_1+u_4^N, u_2+u_4^N).$  It is clear that $f_1:=u_1-u_2$ is an $M$-regular element in $I(\underline{\epsilon})$.  We have $M/f_1M \cong k[[u_2, u_3, u_4]]/(u_2^2, u_2u_3).$ Hence $f_2:=u_2+u_4^N$ is an  $M/f_1M$-regular element in $I(\underline{\epsilon}).$ Therefore,  
$$\depth (I(\underline{\epsilon}), M)=2 >1=\depth (I,M).$$}
\end{example}

Now we state the main result of this section.

\begin{theorem} \label{T:1bs} Let $I$ be a proper ideal of $R$.  The following statements are true. 
\begin{itemize}
\item[{\rm (a)}] Each sequential sequence of $M$ in $I$ can be extended to a maximal one, all maximal sequential sequences of $M$ in $I$ have the same length. This common length is equal to $\E-depth_R(I, M)$.

\item[{\rm (b)}] $\Ef-depth_R(I,M)<\infty$ if and only if there is a maximal sequential f-sequence of $M$ in $I$. If $\Ef-depth_R(I,M)<\infty$ then all maximal sequential f-sequences of $M$ in $I$ have the same length and this common length is equal to $\Ef-depth_R(I,M)$. If $\dim (M/IM)\geq 2$ then $\Ef-depth_R(I,M)<\infty$. If $\dim (M/IM)=0$ then $\Ef-depth_R(I,M)=\infty.$ 

\item[{\rm (c)}] Let $x_1, \ldots , x_n$  be a system  of generators of $I$. Then there exists an integer $N>0$ such that  $\E-depth (I, M)\leq \E-depth(I(\underline{\epsilon}),M)$ and  $\Ef-depth (I, M)\leq \Ef-depth(I(\underline{\epsilon}),M)$ for all $n$-tuples $\underline{\epsilon}=(\epsilon_1, \ldots , \epsilon_n)$ of elements in $\frak m^N$. 
\end{itemize}
\end{theorem}

\begin{proof} By Remark \ref{R:1}(a), we can assume that $R$ is complete. 

(a) Let $x_1, \ldots , x_r$ be a maximal sequential sequence of $M$ in $I$. We prove by induction on $r$ that $r=\E-depth_R(I, M)$.  Let $r=0$. Then there is no sequential element of $M$ in $I$. By Lemma \ref{L:1a}(a), there exists $i\in\Defi(M)$ such that $\depth_R(I, K^i(M))=0$. Hence $\E-depth_R(I, M)=0.$ 

Let $r>0.$ Note that $x_2, \ldots x_r$ is a maximal sequential sequence of $M/x_1M$ in $I.$  By induction, $\E-depth_R(I, M/x_1M)=r-1.$ Let $i>1.$ By Lemma \ref{L:1a}(a), $i\in\Defi(M)$ if and only if $i-1\in\Defi (M/x_1M).$  If $i\in\Defi(M)$ then  $x_1$ is a regular element of $K^i(M)$ and  $\depth_R (I, K^{i-1}(M/x_1M))=\depth_R(I, K^i(M))-1$. Note that if $1\in\Defi(M)$ then $\depth_R(I, K^1(M))\geq 1.$ Therefore, 
$$r-1=\E-depth_R(I, M/x_1M)=\E-depth_R(I,M)-1.$$ Hence, $\E-depth_R(I,M)=r$.    

(b) Suppose that there is  a maximal sequential f-sequence $x_1, \ldots , x_r$ of $M$ in $I.$ We prove that by induction on $r$  that $\Ef-depth_R(I, M)=r$. Let $r=0$. Then there is no sequential f-element of $M$ in $I$. So, $\f-depth_R(I, K^i(M))=0$ for some $i>1.$ Hence $\Ef-depth_R(I, M)=0,$ the result is true for $r=0$. Let $r>0.$ Set $\overline M:=M/x_1M$. Then $x_2, \ldots x_r$ is a maximal sequential f-sequence of $\overline M$ in $I.$  By induction, $\Ef-depth_R(I, \overline M)=r-1.$ It means that $\f-depth (I, K^i(\overline M))\geq \min\{i-1,r-1\}$ for all $i>1$ and there exists $i_0\geq r+1$ such that $\f-depth (I, K^{i_0}(\overline M))=r-1$. Now, let $i>1.$ By Lemma \ref{L:1a}(c),  $x_1$ is an f-element of $K^{i+1}(M)$ and there is an exact sequence $$0\to K^{i+1}(M)/x_1K^{i+1}(M)\to K^i(\overline M)\to (0:_{K^i(M)}x_1)\to 0,$$ where  $\dim_R(0:_{K^i(M)}x_1)\leq 0$. It follows that $H^j_I(K^{i+1}(M)/x_1K^{i+1}(M))$  is Artinian  if and only if $H^j_I(K^i(\overline M))$ is Artinian for all $j\geq 0.$ So, we get by R. Lu and Z. Tang \cite[Theorem 3.10]{LT} that 
$$\f-depth (I, K^{i+1}(M)/x_1K^{i+1}(M))=\f-depth (I, K^i(\overline M)).$$ 
Since $\f-depth (I, K^i(\overline M))\geq \min\{i-1,r-1\}$, we have $\f-depth (I, K^{i+1}(M))\geq \min\{i,r\}$ for all $i>1$. Moreover, $\f-depth (I, K^{i_0+1}(M))=\f-depth (I, K^{i_0}(\overline M))+1=r$ for some $i_0\geq r+1$. Note that $\f-depth_R(I, K^2(M))\geq 1$ since $x_1$ is an f-element of $K^2(M)$. Thus,  $\Ef-depth_R(I, M)=r$.

Suppose that there does not exist a maximal sequential f-sequence of $M$ in $I$. Then for any $r\in\Bbb N$,  there exists a sequential f-sequence of $M$ in $I$ of length $r$. So, it follows by Lemma \ref{L:1a}(d) that $\Ef-depth_R(I, M)=\infty$.

Now we consider three cases.

\noindent{\it Case 1: $\dim_R(M/IM))\geq 2$}. By Remark \ref{R:1}(b), each sequential f-sequence of $M$ is a generalized regular sequence. Since $\dim_R(M/IM))\geq 2$, the length of any generalized regular sequence of $M$ in $I$ is at most $d-2.$ Therefore,  there always exists a maximal sequential f-sequence of $M$ in $I$ and its length is equal to $\Ef-depth_R(I, M)$.

\noindent{\it Case 2: $\dim_R(M/IM))=0$}. Let $i>1$ be an integer. Note that $\dim_R(K^i(M)/IK^i(M))\leq 0$ as $\Ann_R(M)\subseteq \Ann_RK^i(M)$.  Hence $\f-depth_R(I, K^i(M))=\infty$ for all $i>1.$ Therefore, $\Ef-depth_R(I,M)=\infty$.  

\noindent{\it Case 3: $\dim_R(M/IM))=1$}. If there exists a maximal sequential f-sequence of $M$ in $I$ of length $r$ then $\Ef-depth_R(I, M)=r$.  Otherwise, $\Ef-depth_R(I, M)=\infty$.

(c) For each $i\in\Defi(M)$, there exists by Lemma \ref{L:1d} an integer $N_i>0$ such that $\depth (I, K^i(M))\leq \depth(I(\underline{\epsilon}), K^i(M))$ for any $n$-tuple $\underline{\epsilon}=(\epsilon_1, \ldots , \epsilon_n)$ of elements in $\frak m^{N_i}.$ Set $N=\max\{N_i\mid i\in\Defi(M)\}.$ Then $\E-depth_R(I,M)\leq \E-depth_R(I(\underline{\epsilon}),M)$ for any $n$-tuple $\underline{\epsilon}=(\epsilon_1, \ldots , \epsilon_n)$ of elements in $\frak m^{N}.$ Similarly, the rest statement follows by using Lemma \ref{L:1d} for the f-depth of $K^i(M)$ in $I$ for $i>1$.  
\end{proof}

When $\dim_R(M/IM)=1$, $\Ef-depth_R(I, M)$ may be infinite or finite. 

\begin{example} \label{E:1c} {\rm Let $R=K[[u_1,u_2,u_3]]$ be the formal power series ring over a field $K$}.
\begin{itemize}
\item[{\rm (a)}] {\rm Let  $M=R/(u_1)\cap (u_2, u_3)$ and $I=(u_2,u_3)R$. Then $\dim_R(M)=2,$ $\dim (M/IM)=1$ and $\dim (K^2(M)/IK^2(M))=0$. Hence $\f-depth_R(I,K^2(R))=\infty$, so $\Ef-depth_R(I,R)=\infty$.}
\item[{\rm (b)}] {\rm Let $M=(u_2,u_3)R$ and $I=(u_2,u_3)R.$ Then $\dim_R(M)=3$ and $\dim_R(M/IM)=1.$ Note that $\Ass_RK^2(M)=\{(u_2,u_3)R\}.$ Hence $\f-depth_R(I, K^2(M))=0$ and there does not exist a sequential f-element of $M$ in $I$. Therefore, $\Ef-depth_R(I, M)=0.$} 
\end{itemize}
\end{example}

\section{Sequentially Cohen-Macaulay locus}
 
Denote by  $\nCM(M)$ and $\nSCM(M)$ respectively the non Cohen-Macaulay locus and the non sequentially Cohen-Macaulay locus of $M$, i.e
\begin{align}\nCM(M)&=\{\frak p\in\Spec(R)\mid M_{\frak p}\ \text{is not  Cohen-Macaulay}\};\notag\\
\nSCM(M)&=\{\frak p\in\Spec(R)\mid M_{\frak p}\ \text{is not sequentially Cohen-Macaulay}\}.\notag\end{align}
 
In general, $\nCM(M)$ is not a closed subset of $\Spec(R)$ under Zariski topology, see \cite[Example 3.1]{BS1}. If $R$ is a quotient of a Gorenstein local ring then $\nCM(M)$ is closed (see \cite{Sc1}) and the dimension of $\nCM(M)$ was investigated by N. T. Cuong \cite{C}, \cite{C1}. 

Throughout this section, we use the following notations.
\begin{notation} \label{N:2} {\rm For a finitely generated $R$-module $M$ of dimension $d$, we denote by $U_M(0)$  the largest submodule of $M$ of dimension less than $d$. Set $\frak a_i(M):=\Ann_RH^i_{\frak m}(M)$ for all integers $i\leq d$ and  
$\frak a(M):=\frak a_0(M). \frak a_1(M)\ldots \frak a_{d-1}(M).$
Set $D(M):=\{\dim (R/\frak p)\mid \frak \p\in\Ass_R(M)\}.$ We note that if $R$ is a quotient of a Gorenstein local ring then $\dim_RK^i(M)\leq i$  for all $i$. Moreover, $\dim_RK^i(M)=i$ if and only if $i\in D(M)$.} 
\end{notation}
 
We recall the some results about non Cohen-Macaulay locus and non sequentially Cohen-Macaulay locus, see \cite{C1, CNN, NNK}. 

\begin{lemma} \label{L:2} Let $U_M(0)$, $\frak a_i(M)$, $\frak a(M)$ be defined as in Notation \ref{N:2}.  Suppose that  $R$ is a quotient of a Cohen-Macaulay local ring. Then
\begin{enumerate}
\item[{\rm (a)}] $\nCM(M)=\underset{0\leq i<j\leq d}{\bigcup}\Var \big(\frak a_i(M)+\frak a_j(M)\big).$ In particular, $\nCM(M)$ is  a closed subset of $\Spec(R)$ and $\nCM(M)\subseteq  \Var(\frak a(M))$. The equality holds true if $M$ is equidimensional.
\item[{\rm (b)}] $\dim (R/\frak a(M))=\max\{\dim\nCM(M), \dim_RU_M(0)\}.$ 
\item[{\rm (c)}] $\nCM(M)=\{\frak P\cap R\mid \frak P\in\nCM(\widehat M)\}$ and $\dim\nCM(M)=\dim\nCM(\widehat M).$
\end{enumerate}
\end{lemma}

 \begin{lemma} \label{L:2a} Let $H^0_{\frak m}(M)=D_t\subset \ldots \subset D_1\subset D_0=M$ be the dimension filtration of $M$.
 \begin{enumerate}
 \item[{\rm (a)}] If $R$ is catenary then $\nSCM(M)=\bigcup_{i=1}^t\nCM(D_{i-1}/D_i).$
 \item[{\rm (b)}] If $R$ is a quotient of a Cohen-Macaulay local ring then $\nSCM(M)$ is closed under Zariski topology, $\nSCM(M)=\{\frak P\cap R\mid \frak P\in\nSCM(\widehat M)\}$ and 
$\dim\nSCM(M)=\dim\nSCM(\widehat M).$
 \end{enumerate}
 \end{lemma}
  
  The following theorem is the first main result of this paper that gives the interrelatedness between the non sequentially Cohen-Macaulay locus of $M$ and the support as well as the non Cohen-Macaulay locus of deficiency modules of $M$.   

\begin{theorem} \label{T:1} Let $R$ be a quotient of a Gorenstein local ring. Set $K^i:=K^i(M)$ for all integers $i\geq 0.$ Let $D(M)$ and $U_{K^i}(0)$ be defined as in Notation \ref{N:2}. Then  we have
\begin{enumerate}
\item[{\rm (a)}] $\nSCM(M)=X\cup Y\cup Z$, where  $X=\underset{i\notin D(M)}{\bigcup}\Supp_R(K^i)$, $Y=\underset{i\in D(M)}{\bigcup} \nCM (K^i)$ and $Z=\underset{i\in D(M)}{\bigcup}\Supp_R(U_{K^i}(0))\setminus\Supp_R(K^i\big/U_{K^i}(0))$.    
\item[{\rm (b)}] $\dim\nSCM(M)=\max\{q_1, q_2\},$ where $q_1=\underset{i\notin D(M)}{\max}\dim_R(K^i)$, $q_2=\underset{i\in D(M)}{\max}\dim (R/\frak a(K^i))$.    
\end{enumerate}
\end{theorem}

\begin{proof} (a). Let $\frak p \in \nSCM (M).$ By \cite[Theorem 5.5]{Sch},  there exists $j\leq\dim_{R_{\frak p}}(M_{\frak p})$ such that $K^j(M_{\frak p})\neq 0$ and  $K^j(M_{\frak p})$ is not Cohen-Macaulay of dimension $j$.  Set $i=j+\dim (R/\frak p).$ We devide into two cases.

  {\it Case 1: $i\notin D(M)$}. We have $(K^i)_{\frak p}\cong K^j(M_{\frak p})\neq 0$. Hence $\frak p\in X$.

{\it Case 2: $i\in D(M)$}. Assume that $\dim_{R_{\frak p}} K^j(M_\frak p)=j$. Then $(K^i)_{\frak p}\cong K^j(M_\frak p)$ which are not Cohen-Macaulay. Hence $\frak p \in Y.$ So we can  assume that $\dim_{R_{\frak p}} K^j(M_\frak p)<j.$ If $\frak p \in \Supp (K^i/ U_{K^i}(0))$ then $\frak p\supseteq \frak q$ for some  $\frak q\in\Ass_R(K^i/ U_{K^i}(0)).$ Hence $\dim (R/\frak q)=i$ and $\docao (\frak p/\frak q)\leq \dim_{R_{\frak p}} K^j(M_\frak p)<j.$  As $R$ is catenary, $$i=\dim (R/\frak q)=\dim (R/\frak p)+\docao (\frak p/\frak q)<\dim (R/\frak p)+j=i,$$ this is impossible. Hence $\frak p \notin \Supp_R (K^i/ U_{K^i}(0)).$ Note that $\frak p\in\Supp_R(K^i)$ as $K^j(M_{\frak p})\neq 0$. Hence $\frak p \in \Supp_R U_{K^i}(0),$ i.e $\frak p\in Z.$ 

From Case 1 and Case 2, we have $\nSCM (M)\subseteq X\cup Y\cup Z.$ We prove the converse inclusion. Let $\frak p\in X.$ Then there exists $i\notin D(M)$ such that $\frak p\in\Supp_R (K^i)$. Set $j=i-\dim (R/\frak p).$ Then $K^j(M_{\frak p})\neq 0$ and  $$\dim_{R_{\frak p}}(K^j(M_{\frak p})\leq \dim_R(K^i)-\dim (R/\frak p)<i-\dim (R/\frak p)=j.$$ Therefore $M_{\frak p}$ is not sequentially Cohen-Macaulay by \cite[Theorem 5.5]{Sch}, i.e $\frak p\in\nSCM(M)$.

Let $\frak p\in Y.$ Then $\frak p\in\nCM (K^i)$ for some $i\in D(M)$. Hence  $(K^i)_{\frak p}$ is not Cohen-Macaulay. Set $j=i-\dim (R/\frak p).$ Then $K^j(M_{\frak p})$ is not Cohen-Macaulay. Hence $\frak p\in\nSCM(M)$ by \cite[Theorem 5.5]{Sch}.

Let $\frak p\in Z.$  Then  $\frak p\in\Supp_R(U_{K^i}(0))$ and $\frak p\notin\Supp_R(K^i/U_{K^i}(0))$ for some $i\in D(M)$. Hence $\dim_R(K^i)=i$ and  $(K^i)_{\frak p} \cong (U_{K^i}(0))_{\frak p}.$ Set $j=i-\dim (R/\frak p).$ Then $K^j(M_{\frak p})\neq 0$ and
 $$\dim_{R_{\frak p}}(K^j(M_{\frak p}))=\dim_{R_{\frak p}}((U_{K^i}(0))_{\frak p})\leq \dim_R(U_{K^i}(0))-\dim (R/\frak p)<i-\dim (R/\frak p)=j.$$ Hence $\frak p\in\nSCM(M)$ by \cite[Theorem 5.5]{Sch}.

(b). For each $i\in D(M)$, we have by Lemma \ref{L:2}(b) that 
$$\dim \big(R/\frak a(K^i)\big)=\max\{\dim_R\nCM(K^i), \dim_RU_{K^i}(0)\}.$$  Therefore,  we get by assertion (a) that $\dim \nSCM(M)\leq \max\{q_1, q_2\}$ and $q_1\leq \dim \nSCM(M).$ Let $i\in D(M)$ such that $q_2=\dim \big(R/\frak a(K^i)\big).$ Hence 
$$q_2=\max\{\dim_R\nCM(K^i), \dim_R(U_{K^i}(0))\}.$$ If $q_2=\dim_R\nCM(K^i)$ then we get by assertion (a) that $q_2\leq \dim \nSCM(M).$ So, we assume that $q_2=\dim_R(U_{K^i}(0)).$ Then  $q_2=\dim (R/\frak p)$ for some $\frak p\in \min\Ass_RU_{K^i}(0).$  Hence $H^0_{\frak p R_{\frak p}}\big(U_{K^i}(0)_{\frak p}\big)\neq 0.$ If $\frak p\in Z$ then $\frak p\in\nSCM(M)$ by assertion (a). It follows that  $q_2=\dim (R/\frak p)\leq \dim \nSCM(M)$. So, we assume that $\frak p\notin Z.$ Then $\frak p\in\Supp_R(K^i/U_{K^i}(0)).$  Hence $\frak p\supseteq\frak q$ for some $\frak q\in \Ass_R(K^i/U_{K^i}(0))$. Note that $\dim (R/\frak q)=\dim_R(K^i)=i$ as $i\in D(M)$. Since $\dim (R/\frak p)=\dim_R(U_{K^i}(0))<i,$ we have $\frak p\neq \frak q$. Hence $\dim_{R_{\frak p}}(K^i)_{\frak p}>0.$ As $H^0_{\frak p R_{\frak p}}\big(U_{K^i}(0)_{\frak p}\big)\neq 0,$ it follows that $(K^i)_{\frak p}$ is not Cohen-Macaulay, i.e $\frak p\in Y.$  Therefore, $q_2=\dim (R/\frak p)\leq \dim \nSCM(M).$
\end{proof}

We recall an interesting result of L. Ma, P. H. Quy and I. Smirnov \cite[Theorem 4.6]{MQS} on the  non Cohen-Macaulay locus with respect to small pertubations of an f-sequence. 

\begin{theorem} \label{P:2b} Suppose $R$ is a quotient of a Cohen-Macaulay local ring. Let $x_1, \ldots , x_r$ be an f-sequence of $M$. If $M$ is equidimensional then there exists an integer $N>0$ such that 
$$\dim\nCM(M/(x_1, \ldots , x_r)M)\geq \dim\nCM(M/(x_1+\epsilon_1, \ldots , x_r+\epsilon_r)M)$$ for all $\epsilon_1, \ldots , \epsilon_n\in \frak m^N$.
\end{theorem}

In this paper, we show that the dimension of non sequentially Cohen-Macaulay locus with respect to a sequential f-sequence does not increase under small pertubations. As mentioned in Remark \ref{R:1}(b), sequential f-sequences are not necessarily f-sequences; f-sequences are not necessarily sequential f-sequences. 

\begin{remark}\label{R:2} {\rm Suppose that  $0\to M'\to M\to M''\to 0$ is an exact sequence of finitely generated $R$-modules with $\dim_R(M'')\leq 0$. If either  $\dim (R/\frak a(M))>0$ or $\dim (R/\frak a(M'))>0$ then $\dim_R(M')=\dim_R(M)\geq 2$ and $\dim (R/\frak a(M'))=\dim (R/\frak a(M)).$
If $R$ be a quotient of a Cohen-Macaulay local ring then  $\dim (R/\frak a(M))\leq \dim_R(M)-1$  by Lemma \ref{L:1}(d).}
\end{remark}

The following theorem is the second main result of this paper.

\begin{theorem}\label{T:2} Assume that $R$ is a quotient of  a Cohen-Macaulay local ring and $x_1, \ldots, x_r$ is a sequential f-sequence of $M$. If $M/(x_1, \ldots, x_r)M$ is not sequentially Cohen-Macaulay then exists an integer  $N>0$ such that
$$\dim \nSCM(M/(x_1, \ldots, x_r)M) \geq \dim \nSCM(M/(x_1 + \epsilon_1, \ldots , x_r + \epsilon_r )M)$$
 for every $\epsilon_1, \ldots, \epsilon_r \in \frak m^{N}.$  
\end{theorem}

\begin{proof} By Remark \ref{R:1}(a) and Lemma \ref{L:2a}(b), we can assume that $R$ is complete.

Set $I=(x_1, \ldots , x_r)$. By Proposition \ref{P:1} and \cite[Lemma 1.9]{HT},  there exists an integer $N>0$ such that for any $r$-tuple $\underline{\epsilon}=(\epsilon_1, \ldots , \epsilon_r)$ of elements in $\frak m^N$ we have

(a) $x_1+\epsilon_1, \ldots , x_r+\epsilon_r$ is an $M$-sequential f-sequence;  

(b) $\dim \big(R/(\frak a(K^n(M))+I)\big)\geq \dim \big(R/(\frak a(K^n(M))+I(\underline{\epsilon})\big)$ for all $n$.

 Now, let  $\epsilon_1, \ldots, \epsilon_r \in \frak m^{N}.$  For each integer $1\leq k\leq r$, we set $M_k:=M/(x_1, \ldots , x_k)M$ and $M'_k:=M/(y_1, \ldots , y_k)M,$ where $y_i=x_i+\epsilon_i$ for $i\leq k$.  For each $1\leq k\leq r,$  we have by Lemma \ref{L:1a}(c) the  exact sequences
\begin{align}&0\to K^{n+1}(M_{k-1})/x_kK^{n+1}(M_{k-1})\to  K^n(M_k)\to (0:_{K^n(M_{k-1})}x_k)\to 0\\
& 0\to K^{n+1}(M'_{k-1})/y_kK^{n+1}(M'_{k-1})\to  K^n(M'_k)\to (0:_{K^n(M'_{k-1})}y_k)\to 0\end{align}
for all $n\geq 1$. Moreover, $\dim_R(0:_{K^n(M_{k-1})}x_k)\leq 0$, $\dim_R(0:_{K^n(M'_{k-1})}y_k)\leq 0$, $x_k$ is an f-element of  $K^n(M_{k-1})$ and $y_k$ is an f-element of  $K^n(M'_{k-1})$ for all $n\geq 2$.  

Set $s:=\dim\nSCM(M'_r).$ If $s\leq 0$ then $s\leq \dim\nSCM(M_r)$ as $M_r$ is not sequentially Cohen-Macaulay. So, we can assume that $s>0.$ We get by Theorem \ref{T:1} that $s=\max\{q_1, q_2\}$, where  $q_1=\underset{i\notin D(M'_r)}{\max}\dim_R(K^i(M'_r))$ and
$q_2=\underset{i\in D(M'_r)}{\max}\dim\big(R/\frak a(K^i(M'_r))\big).$  We devide into two cases.

{\it Case 1: $s=q_1$}. Then $s=\dim_RK^i(M'_r)$ for some $i\notin D(M'_r).$ Hence $i>s>0$ and hence $i\geq 2.$  So, $\dim_R(0:_{K^i(M'_{r-1})}y_r)\leq 0$. Since $s>0$, we have by exact sequence (2) for $k=r, n=i$ that $\dim_R(K^{i+1}(M'_{r-1})/y_rK^{i+1}(M'_{r-1}))=s>0.$ Since $y_r$ is an f-element of $K^{i+1}(M'_{r-1})$, it follows that  $y_r$ is a parameter element of $K^{i+1}(M'_{r-1})$. Hence $\dim_R(K^{i+1}(M'_{r-1}))=s+1$. Using exact sequence (2) consecutively for $(k,n)=(r-1,i+1), \ldots , (1,i+r-1)$ with the same arguments as in the above, we have $\dim_R(K^{i+r}(M))=s+r.$ As $x_1$ is an f-element of $K^{i+r}(M)$, it follows that  $x_1$ is a parameter element of $K^{i+r}(M)$. Hence $\dim_R(K^{i+r}(M)/x_1K^{i+r}(M))=s+(r-1)>0.$ As $\dim_R(0:_{K^{i+r-1}(M)}x_1)\leq 0$, we have by exact sequence (1) for $k=1, n=i+r-1$ that $\dim_R(K^{i+r-1}(M_1))=s+(r-1)$. Using exact sequence (1) consecutively for $(k,n)=(2,i+r-2), \ldots , (r,i)$ with the same arguments as in the above, we have $\dim_R(K^i(M_r))=s$. Since $s<i,$ we have $i\notin D(M_r).$ Therefore $\dim\nSCM(M_r)\geq s$ by Theorem \ref{T:1}.

{\it Case 2: $s=q_2$}. Then $s=\dim\big(R/\frak a(K^i(M'_r))\big)$ for some $i\in D(M'_r).$ Since $s>0$, we have by Remark \ref{R:2} that $i=\dim_R(K^i(M'_r))\geq 2.$ Hence $\dim_R(0:_{K^i(M'_{r-1})}y_r)\leq 0$. So, from the exact sequence (2) for $k=r, n=i$ with notice that $s>0$, we get by Remark \ref{R:2} that $\dim_R(K^{i+1}(M'_{r-1})/y_rK^{i+1}(M'_{r-1}))=\dim_R(K^i(M'_r))=i$ and $$\dim\big(R/\frak a(K^{i+1}(M'_{r-1})/y_rK^{i+1}(M'_{r-1}))\big)=s.$$  Because $i>0$ and  $y_r$ is an f-element of $K^{i+1}(M'_{r-1})$, it follows that $y_r$ is a parameter element of $K^{i+1}(M'_{r-1})$. Hence $\dim_R(K^{i+1}(M'_{r-1}))=i+1.$ Set $N=(0:_{K^{i+1}(M'_{r-2})}y_{r-1}).$ By applying the functor $R/y_rR\otimes_R-$ to the exact sequence (2) for $k=r-1, n=i+1$, we get the exact sequence 
\begin{align}\Tor_1^R(R/y_rR; N)\to K^{i+2}(M'_{r-2})/(y_{r-1}, y_r)K^{i+2}(M'_{r-2})&\to  K^{i+1}(M'_{r-1})/y_rK^{i+1}(M'_{r-1})\notag\\
&\to \Tor_0^R(R/y_rR; N).\notag\end{align}
Since $\dim_R(N)\leq 0$, we have  $\dim_R(\Tor_j^R(R/y_rR; N))\leq 0$ for $j=0, 1.$ As $s>0,$ we have $$\dim_R\Big(R/\frak a\big(K^{i+2}(M'_{r-2})/(y_{r-1},y_r)K^{i+2}(M'_{r-2})\big)\Big)=s$$ and $\dim_R\big(K^{i+2}(M'_{r-2})/(y_{r-1}, y_r)K^{i+2}(M'_{r-2})\big)=i.$ Since $i>0$, $y_{r-1}, y_r$ is an f-sequence of $K^{i+2}(M'_{r-2})$ by Lemma \ref{L:1a}(d), hence it is a part of a system of parameters of $K^{i+2}(M'_{r-2})$. So, $\dim_R(K^{i+2}(M'_{r-2}))=i+2.$
Applying the functor $R/(y_{k+1}, \ldots ,y_r)R\otimes_R-$ to the exact sequence (2) consecutively for $(k,n)=(r-2,i+2), \ldots , (1,i+r-1)$ with similar arguments as in the above, we have 
$$\dim_R\Big(R/\frak a\big(K^{i+r}(M)/(y_1, \ldots ,y_r)K^{i+r}(M)\big)\Big)=s$$
 and  $\dim_R(K^{i+r}(M))=i+r.$ Since $i+r\geq r+2$, we have by Lemma \ref{L:1a}(d) that $x_1, \ldots , x_r$ and $y_1, \ldots ,y_r$ are f-sequences of $K^{i+r}(M)$. So, we get by \cite[Proposition 4.4]{MQS}  that 
\begin{align}\dim\big(R/\frak a(K^{i+r}(M)/IK^{i+r}(M)))\big)&=\dim\big(R/(I+\frak a(K^{i+r}(M)))\big)\notag\\
&\geq \dim\big(R/(I(\underline{\epsilon})+\frak a(K^{i+r}(M))\big)\notag\\
&=\dim\big(R/\frak a(K^{i+r}(M)/I(\underline{\epsilon})K^{i+r}(M))\big)=s.\notag\end{align}

Set $P:=(0:_{K^{i+r-1}(M)}x_1)$. Using the functor $R/(x_2, \ldots , x_r)R\otimes_R-$ to the exact sequence (1) for $k=1, n=i+r-1$, we get the exact sequence
\begin{align}\Tor_1^R(R/(x_2, \ldots, x_r)R; P)\to K^{i+r}(M)/IK^{i+r}(M)&\to  K^{i+r-1}(M_1)/(x_2, \ldots , x_r)K^{i+r-1}(M_1)\notag\\
&\to \Tor_0^R(R/(x_2, \ldots , x_r)R; P).\notag\end{align}
Since $\dim_R(P)\leq 0$, we have  $\dim_R(\Tor_j^R(R/(x_2, \ldots , x_r)R; N'))\leq 0$ for $j=0, 1.$ Since $\dim_R\Big(R/\frak a\big(K^{i+r}(M)/IK^{i+r}(M)\big)\Big)\geq s>0$, we get by Remark \ref{R:2} that 
$$\dim_R\Big(R/\frak a\big(K^{i+r-1}(M_1)/(x_2, \ldots , x_r)K^{i+r-1}(M_1)\big)\Big)\geq s>0.$$
 Moreover, since $x_1$ is an f-element of $K^{i+r}(M)$, we get by exact sequence (1) for $k=1$ and $n=i+r-1$ that 
$$\dim_RK^{i+r-1}(M_1)=\dim_R(K^{i+r}(M)/x_1K^{i+r}(M))=\dim_R(K^{i+r}(M))-1=i+r-1.$$ By applying  the functor $R/(x_{k+1}, \ldots ,y_r)R\otimes_R-$ to the exact sequence (2) consecutively for $(k,n)=(2,i+r-2), \ldots , (r,i)$ with similar arguments as in the above, we have 
$$\dim\big(R/\frak a(K^i(M_r))\big)\geq s.$$
 Moreover, since $\dim_R(K^{i+r}(M))=i+r$, we have by exact sequence (1) that
$$\dim_R(K^i(M_r))=\dim_R(K^{i+1}(M_{r-1}))-1=\ldots =\dim_R(K^{i+r}(M))-r=i.$$ Hence $i\in D(M_r).$ Therefore,  $\dim\nSCM(M_r)\geq \dim\big(R/\frak a(K^i(M_r))\big)\geq s$ by Theorem \ref{T:1}. 
\end{proof}

As a consequence, we show that the sequential Cohen-Macaulayness with respect to a sequential sequence is preserved under small pertubations. 

\begin{corollary} \label{C:2} Assume that $R$ is a quotient of  a Cohen-Macaulay local ring and $x_1, \ldots, x_r$ is an $M$-sequential sequence. If $M/(x_1, \ldots, x_r)M$ is  sequentially Cohen-Macaulay then exists an integer  $N>0$ such that
$M/(x_1 + \epsilon_1, \ldots , x_r + \epsilon_r )M)$  is  also sequentially Cohen-Macaulay 
 for all $\epsilon_1, \ldots, \epsilon_r \in \frak m^{N}.$  
\end{corollary}

\begin{proof} By Remark \ref{R:1}(a), we can assume that $R$ is complete. By Proposition \ref{P:1}, there exists an integer $N>0$ such that $x_1+\epsilon_1, \ldots , x_r+\epsilon_r$ is an $M$-sequential sequence for any $r$-tuple $\underline{\epsilon}=(\epsilon_1, \ldots , \epsilon_r)$ of elements in $\frak m^N$. Set $M_r=M/(x_1, \ldots , x_r)M$.  Since $M_r$ is sequentially Cohen-Macaulay, we have $d-r=\dim_R(M_r)=\E-depth_R(M_r).$ Since $x_1, \ldots, x_r$ is an $M$-sequential sequence, we have $\E-depth_R(M)=d,$ i.e. $M$ is sequentially Cohen-Macaulay. For any $\epsilon_1, \ldots, \epsilon_r \in \frak m^{N},$ since $x_1+\epsilon_1, \ldots , x_r+\epsilon_r$ is an $M$-sequential sequence, we have 
$$\E-depth_R(M/(x_1 + \epsilon_1, \ldots , x_r + \epsilon_r )M))=d-r=\dim_R(M/(x_1 + \epsilon_1, \ldots , x_r + \epsilon_r )M)),$$  i.e. $M/(x_1 + \epsilon_1, \ldots , x_r + \epsilon_r )M)$ is  sequentially Cohen-Macaulay.  
\end{proof}

Finally we give an example to clarify the results.

\begin{example} \label{E:2} {\rm Let $d\geq 3$ be an integer and  $R=K[[u_1,\ldots,u_{d+2}]]/(u_1,u_2)\cap (u_3,u_4)$, where $K[[u_1,\ldots,u_{d+2}]]$ is the formal power series ring over a field $K$. Let $f=u_1+u_3$ and $g=u_1+u_5.$  Set $\frak p=(u_1, u_2, u_3, u_4)R$ and $\frak m=(u_1,\ldots,u_{d+2})R$. Then

  (a) $\dim (R)=d$, $K^d(R)$ is Cohen-Macaulay of dimension $d$, $K^{d-1}(R)$ is Cohen-Macaulay of dimension $d-2$ and  $K^i(R)=0$ for all $i<d-1.$ Therefore, $R$ is not sequentially Cohen-Macaulay and we get by Theorem \ref{T:1} that $$\dim\nSCM(R)=\dim_RK^{d-1}(R)=d-2.$$

  (b) It is clear that $f$ is a regular element of $R$. Since $f\in \frak p$ and $\frak p\in\Att_RH^{d-1}_{\frak m}(R)$ with $\dim (R/\frak p)=d-2>0$, it follows that $f$ is not a sequential f-element of $R$. We have $K^{d-1}(R/fR)$ is is Cohen-Macaulay of dimension $d-1$, $K^{d-2}(R/fR)$ is is Cohen-Macaulay of dimension $d-2$ and $K^i(R/fR)=0$ for all $i<d-2$. Therefore, $R/fR$ is sequentially Cohen-Macaulay by \cite[Theorem 5.5]{Sch}. For any integer $N>0$, take $\epsilon=u_5^N\in\frak m^N$, then $f+\epsilon$ is a sequential element of $R$. It follows by Lemma \ref{L:1a}(a) and Theorem \ref{T:1} that $$\dim\nSCM(R/(f+\epsilon)R)=\dim\nSCM(R)-1=d-3\geq 0.$$ In particular, $R/(f+\epsilon)R$ is not sequentially Cohen-Macaulay, thus the pertubation ring $R/(f+\epsilon)R$ is, in some sense, worse than the original ring $R/fR$.

 (c) It is clear that $g$ is a regular element and a sequential element of $R$. So, it follows by Lemma \ref{L:1a}(a) and Theorem \ref{T:1} that $\dim\nSCM(R/gR)=\dim\nSCM(R)-1=d-3.$ By Theorem \ref{T:2}, there exists an integer $N>0$ such that $\dim\nSCM(R/(g+\epsilon)R)\leq d-3$ for any $\epsilon\in\frak m^N$.}
\end{example}

\end{document}